\documentclass[a4paper,12pt, reqno]{amsart}

\usepackage{amsfonts, color}
\usepackage[textwidth=460pt]{geometry}
\usepackage{amstext, amsthm, amssymb, amsmath, mathtools}
\usepackage{mathrsfs}
\usepackage{graphicx}
\usepackage{color}
\usepackage{lineno}
\usepackage[linesnumbered,lined,ruled,commentsnumbered]{algorithm2e}

\usepackage{hyperref, cleveref}
\hypersetup{colorlinks = true, citecolor = blue}

\theoremstyle{plain}
\begingroup
\newtheorem{theorem}{Theorem}[section]
\newtheorem{lemma}[theorem]{Lemma}
\newtheorem{proposition}[theorem]{Proposition}
\newtheorem{corollary}[theorem]{Corollary}
\endgroup

\theoremstyle{definition}
\begingroup
\newtheorem{defn}{Definition}
\newtheorem{example}{Example}

\endgroup

\theoremstyle{remark}
\newtheorem{remark}{Remark}

\numberwithin{equation}{section}

\newcommand{\R}{\mathbb{R}}

\newcommand{\M}{\mathcal{M}}

\newcommand{\NN}{\mathbb{N}}

\newcommand{\Rext}{\mathbb{R}\cup\{+\infty\}}

\crefname{theorem}{Theorem}{Theorems}
\crefname{lemma}{Lemma}{Lemmas}
\crefname{corollary}{Corollary}{Corollaries}
\crefname{section}{Section}{Sections}
\crefname{proposition}{Proposition}{Proposition}
\crefname{defn}{Definition}{Definitions}
\crefname{remark}{Remark}{Remarks}
\crefname{conjecture}{Conjecture}{Conjectures}
\crefname{example}{Example}{Examples}
\crefname{assumption}{Assumption}{Assumptions}

\title[Trade-off invariance for weighted scalarizations]
{Trade-off invariance for weighted scalarizations in multi-objective optimization}

\author[J.~Klemenc]{Jona Klemenc$^{1,2}$}
\author[A.~Scagliotti]{Alessandro Scagliotti$^{1,2}$}

\thanks{$^1$CIT School, Technical University of Munich, Garching bei M\"unchen, Germany.}
\thanks{$^2$Munich Center for Machine Learning (MCML), Munich, Germany.}
\thanks{Email addresses: \texttt{jona.klemenc@tum.de}, \texttt{scag@ma.tum.de}.}

\newcommand{\FundingLogos}{%
  \raisebox{0pt}{\includegraphics[height=1.5cm]{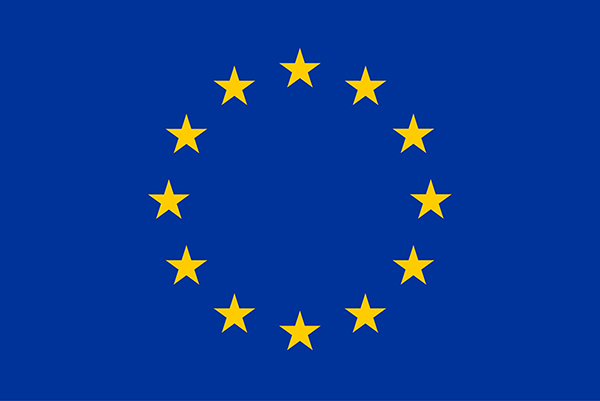}}%
  \hspace{1em}%
  \raisebox{0pt}{\includegraphics[height=1.5cm]{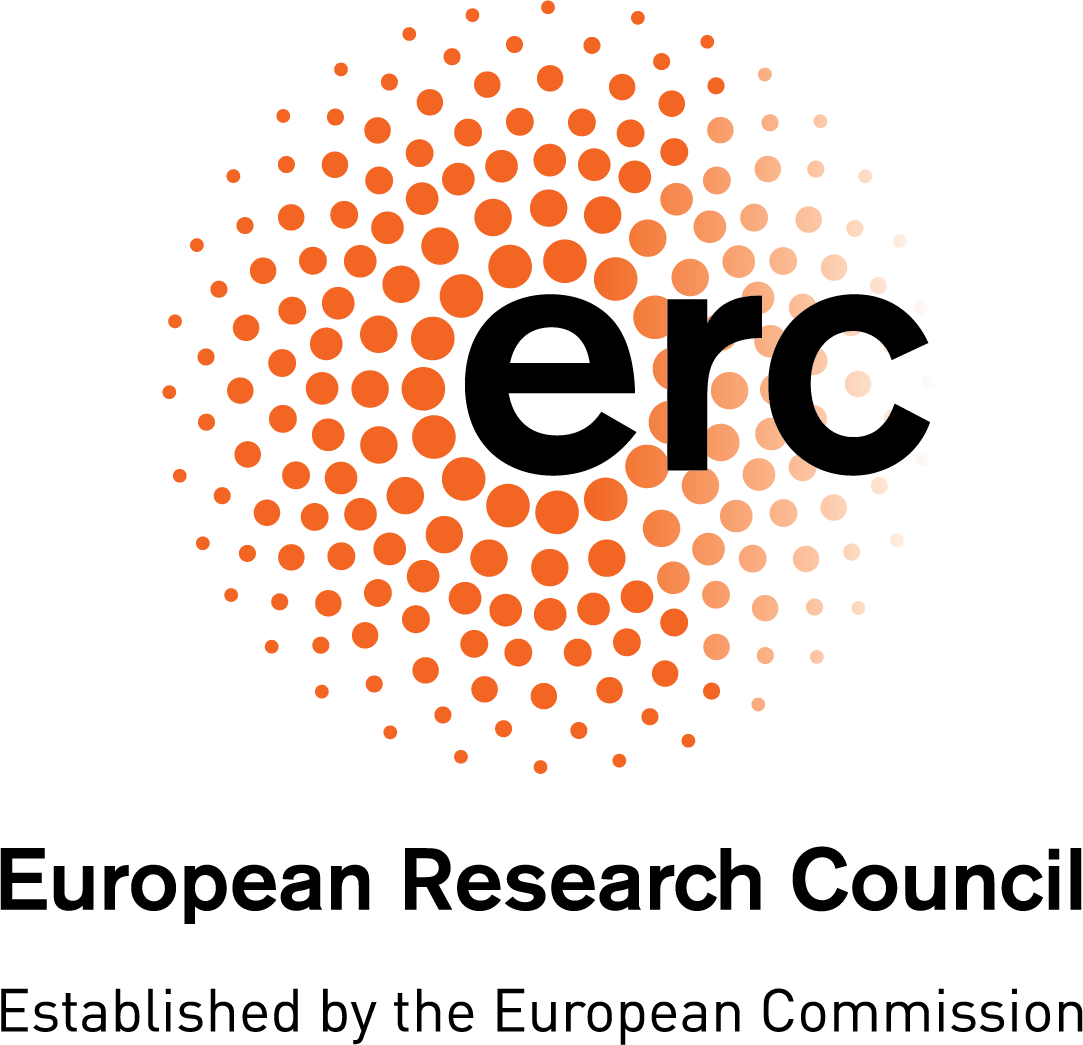}}%
}

\usepackage{tcolorbox}
\tcbset{colback=white,boxrule=0.5pt,arc=2pt,left=4pt,right=4pt,top=4pt,bottom=4pt}

\begin{document}


\begin{abstract}
We consider weighted-sum scalarizations for an abstract multi-objective minimization problem defined by the vector-valued map $U\ni u\mapsto \big( f_1(u),\ldots, f_N(u)\big)$,
where $U$ is an arbitrary nonempty set and no topology, convexity, compactness, or lower semicontinuity assumption is imposed. Using the open simplex as parameter space for positive weights, we show that the Trade-off Invariance Principle for scalarizations yields a generic uniqueness property in the objective space. 
Namely, for almost every weight vector, all minimizers of the corresponding weighted-sum scalarization have the same objective vector, although the minimizers themselves need not be unique. Moreover, excluding again a null-measure subset, all minimizing sequences determine the same limiting objective vector, independently of the chosen sequence.
We also give a geometric interpretation of these results in the attainable objective set: for almost every positive weight vector, the scalarization exposes at most one nondominated point.
Furthermore, minimizing sequences determine at most one asymptotically exposed objective vector in the closure of the attainable set.
\end{abstract}

\maketitle

\section{Introduction}

Weighted-sum scalarization is one of the most classical tools in multi-objective optimization. Given several competing objective functionals, one replaces the vector-valued problem by a family of scalar minimization problems obtained as positive linear combinations of the objectives. This viewpoint is standard in the theory of multicriteria optimization and is closely related to the geometric picture of nondominated points in the objective space; see, for instance, \cite{Ehrgott2005,Miettinen1998}.

In this paper, we consider a multi-objective minimization problem of the form
\begin{equation} \label{eq:intro_MOO}
    \min_{u\in U} f(u) =  \min_{u\in U}\big(f_1(u),\dots,f_N(u)\big),
\end{equation}
and we investigate a basic question about weighted-sum scalarizations
\begin{equation}\label{eq:intro-weighted-sum}
    H_\lambda(u)
    \coloneqq
    \sum_{i=1}^N \lambda_i f_i(u),
    \qquad u\in U,
\end{equation}
in a completely abstract setting.
Here, $U$ is an arbitrary nonempty set, the objectives $f_i$ take values in $\R\cup\{+\infty\}$, and $\lambda \in (0, +\infty)^N$ is the weight vector. No topology, metric, convexity, compactness, or lower semicontinuity assumption is imposed. For a fixed positive weight vector, one may ask whether different minimizers of the corresponding weighted-sum scalarization can yield different objective vectors. More generally, even when minimizers do not exist, one may ask whether different minimizing sequences can produce different limiting objective values. In other words, is the trade-off selected by the scalarization intrinsic to the weight vector, or can it depend on the particular minimizing sequence?

Our motivation comes from the recent Trade-off Invariance Principle introduced in \cite{Fornasier2026} for regularized functionals of the form $F+\sum_{j=1}^d \alpha_j G_j$ with $(\alpha_1,\ldots,\alpha_d)\in (0,\infty)^d$. In that framework, it was shown that, outside a negligible exceptional set of parameters $E\subset (0,\infty)^d$, all minimizers have the same trade-off, and all minimizing sequences determine the same limiting value of each non-reference term. The aim of the present work is to transfer that principle to the symmetric multi-objective framework of weighted-sum scalarizations.

A first point is that weighted-sum scalarizations $\sum_{i=1}^N\lambda_i f_i$ of \cref{eq:intro_MOO} are naturally parametrized by the simplex of positive weights $\Delta_N^\circ$, rather than the whole positive orthant. After choosing any objective as a reference term, the weighted sum can be rewritten in the form treated in \cite{Fornasier2026}. This reduction allows us to transport the negligible exceptional sets from the anchored coordinates to the simplex of weights. 
On this basis, in \cref{thm:objective-invariance-minimizers} we prove that, for almost every positive weight vector, all minimizers of the scalarized problem have the same objective vector.
{Since minimizers of a weighted-sum scalarization with strictly positive weights are Pareto optimal, this means that, generically, the scalarization selects at most one nondominated objective vector, even though the corresponding efficient solutions in the decision space need not be unique.}
Then, in \cref{thm:objective-invariance-min-seq} we establish the analogous statement for minimizing sequences, showing that, outside another negligible set, the limiting objective vector is uniquely determined by the scalarization whenever the scalarized infimum is not $-\infty$.
{Thus, the property of being generically single-valued in the objective space persists at the level of approximate minimization.
Finally, we stress the fact that, in the absence of suitable convexity assumptions, varying the positive coefficients of the weighted-sum scalarization does not necessarily recover the entire Pareto front. Indeed, nondominated objective vectors that are not supported by a linear scalarization may remain undetected. Accordingly, our results concern the generic uniqueness of the objective vector selected by a given weighted-sum scalarization, rather than the ability of this scalarization method to approximate the full Pareto front.}

These results admit a natural geometric interpretation in the attainable objective set {$f(\operatorname{dom} f)$}. For a fixed weight vector $\lambda \in \Delta_N^\circ$, the scalarization selects the subset of attainable objective vectors on which the corresponding linear functional attains its minimum{, which consists of nondominated objective vectors belonging to the Pareto front}. 
Our minimizer result shows that, for almost every positive weight vector, this exposed subset contains at most one point. {Equivalently, we can say that the nondominated objective vector selected by the scalarization is generically unique, even though the corresponding efficient solutions in the decision space need not be unique.}
In the finite-valued setting (i.e., $f_i(u)\neq +\infty$ for every $u\in U$ and every $i=1,\ldots,N$), the minimizing-sequence result yields an asymptotic counterpart: minimizing sequences determine at most one asymptotically exposed objective vector in the closure of the attainable set. Thus, generically, weighted-sum scalarization selects a unique objective-space trade-off not only at the exact level, but also at the asymptotic level.

The paper is organized as follows. In \cref{section:abstract-setup} we introduce the abstract multi-objective setting, the simplex parametrization of positive weights, and the reduction to the anchored framework. In \cref{section:generic-uniqueness} we prove the generic uniqueness statements for minimizers and minimizing sequences, and provide their geometric interpretation in the attainable objective set, including the notion of asymptotically exposed objective vector.

\section{Abstract setup for multi-objective optimization} \label{section:abstract-setup}

We consider an abstract multi-objective minimization problem of the form
\begin{equation}\label{eq:MOO}
    \min_{u\in U} f(u)
    =
    \min_{u\in U}
    \big(f_1(u),\dots,f_N(u)\big),
\end{equation}
where $N\geq 2$, $U\neq\varnothing$ is an arbitrary set, and $f_i\colon U\to \Rext$ for $i=1,\dots,N$, without further assumptions and structure on $U$ or on the objective functionals.
The purpose of this abstract formulation is to isolate a purely variational feature of weighted-sum scalarizations, independently of the analytic structure of the underlying problem.

We assume only that the multi-objective problem is not identically infeasible, namely
\begin{equation*}
    \operatorname{dom} f
    \coloneqq
    \big\{
        u\in U:
        f_i(u)<+\infty
        \text{ for every } i=1,\dots,N
    \big\}
    \neq \varnothing.
\end{equation*}
{We briefly recall the standard terminology of multi-objective optimization. We define the attainable objective set by
\begin{equation} \label{eq:def_attainable}
    Y
    \coloneqq
    f(\operatorname{dom}f)
    \subset \R^N.
\end{equation}
On $\R^N$, we use the componentwise order. Namely, for $y,z\in\R^N$, we write $y\leq z$ if $y_i\leq z_i$ for every $i=1,\dots,N$, and $y<z$ if $y_i<z_i$ for every $i=1,\dots,N$.
An element $\bar u\in\operatorname{dom}f$ is called \emph{efficient}, or \emph{Pareto optimal}, if there is no $u\in\operatorname{dom}f$ such that $f(u)\leq f(\bar u)$ and $f(u)\neq f(\bar u)$. In this case, $f(\bar u)$ is called a \emph{nondominated objective vector}. Equivalently, $y\in Y$ is nondominated if there is no $z\in Y$ such that $z\leq y$ and $z\neq y$. The set
\begin{equation*}
    Y_N
    \coloneqq
    \big\{
        y\in Y:
        \mbox{there is no }z\in Y
        \mbox{ such that }z\leq y
        \mbox{ and }z\neq y
    \big\}
\end{equation*}
is called the \emph{nondominated set}, or \emph{Pareto front}.
Similarly, $\bar u\in\operatorname{dom}f$ is called \emph{weakly efficient}, or \emph{weakly Pareto optimal}, if there is no $u\in\operatorname{dom}f$ such that $f(u) < f(\bar u)$. The corresponding objective vector is called \emph{weakly nondominated}, and the set of all weakly nondominated objective vectors is denoted by $Y_{wN}$.
Following the terminology of \cite{Ehrgott2005}, efficiency and Pareto optimality refer to feasible elements of the decision space, whereas nondominance refers to their objective vectors.}

Since only the relative proportions of the weights are relevant, we parametrize weighted scalarizations by the open simplex
\begin{equation}\label{eq:open-simplex}
    \Delta_N^\circ
    \coloneqq
    \biggl\{
        \lambda\in(0,+\infty)^N:
        \sum_{i=1}^N \lambda_i =1
    \biggr\}.
\end{equation}
For $\lambda=(\lambda_1,\dots,\lambda_N)\in\Delta_N^\circ$, we define the weighted-sum scalarization
\begin{equation}\label{eq:weighted-sum}
    H_\lambda(u)
    \coloneqq
    \sum_{i=1}^N \lambda_i f_i(u),
    \qquad u\in U.
\end{equation}
Its value function is
\begin{equation}\label{eq:value-function}
    V(\lambda)
    \coloneqq
    \inf_{u\in U} H_\lambda(u)
    =
    \inf_{u\in U}\sum_{i=1}^N \lambda_i f_i(u),
\end{equation}
and the set of minimizers is
\begin{equation}\label{eq:set_minimizers}
    \M_\lambda
    \coloneqq
    \operatorname*{argmin}_{u\in U} H_\lambda(u).
\end{equation}
Thus $u\in\M_\lambda$ if and only if $H_\lambda(u)=V(\lambda)$.
We shall also consider minimizing sequences: a sequence $(u_k)_{k\in\NN}\subset U$ is called a minimizing sequence for $H_\lambda$ if $H_\lambda(u_k)\to V(\lambda)$ as $k\to\infty$.

The guiding question is whether non-uniqueness for the scalarized problem may correspond to non-uniqueness in the objective space. More precisely, if $\M_\lambda$ contains several elements, can one have $u_1,u_2\in\M_\lambda$ with $f(u_1)\neq f(u_2)$?
Similarly, if different minimizing sequences are considered, are the limiting objective values intrinsic to the scalarization, or can they depend on the particular sequence chosen?

\subsection{Equivalent parametrization by the positive orthant}

Although we use $\Delta_N^\circ$ as the primary parameter space, it is sometimes convenient to work on the whole positive orthant $(0,+\infty)^N$.
The reason is that two positive weight vectors lying on the same ray define the same scalarized minimization problem.

\begin{lemma}[Invariance under positive rescaling]\label{lem:ray-invariance}
Let $\lambda\in(0,+\infty)^N$ and let $c>0$.
Then the scalarized functionals $H_\lambda$ and $H_{c\lambda}$ have the same minimizers and the same minimizing sequences.
\end{lemma}

\begin{proof}
For every $u\in U$ one has $H_{c\lambda}(u)=c\,H_\lambda(u)$.
Since $c>0$, multiplication by $c$ preserves minimizers and minimizing sequences.
\end{proof}

Hence the scalarized problem depends only on the positive ray $[\lambda] \coloneqq \{c\lambda:c>0\}$. 
The simplex $\Delta_N^\circ$ is precisely a section of the positive rays, since every $\lambda\in(0,+\infty)^N$ admits the unique normalization
\begin{equation*}
    \widehat\lambda
    \coloneqq
    \frac{\lambda}{\sum_{i=1}^N\lambda_i}
    \in\Delta_N^\circ.
\end{equation*}

The trade-off invariance principle in \cite{Fornasier2026} is formulated for families of the form $F+\sum_{j=1}^m \alpha_j G_j$, where one term is singled out as a reference.
The next lemma shows that, after choosing any component as reference, the simplex parametrization is fully compatible with that framework.

\begin{lemma}[Reduction to the anchored form]\label{lem:normalization}
Let $\lambda\in\Delta_N^\circ$, and fix $k\in\{1,\dots,N\}$.
Define $\alpha_i
    \coloneqq
    {\lambda_i}/{\lambda_k}$ for every $i\neq k$.\\
Then the scalarized functional $H_\lambda = \sum_{i=1}^N \lambda_i f_i$ has the same minimizers and the same minimizing sequences as $\widetilde H_\alpha^{(k)}
    \coloneqq
    f_k +\sum_{i\neq k}\alpha_i f_i$.
\end{lemma}

\begin{proof}
For every $u\in U$, we have $H_\lambda(u)
    =
    \lambda_k
    \big(
        f_k(u)+\sum_{i\neq k}\alpha_i f_i(u)
    \big)
    =
    \lambda_k \widetilde H_\alpha^{(k)}(u)$. 
Since $\lambda_k>0$, the conclusion follows.
\end{proof}

In particular, \cref{lem:normalization} shows that the symmetric weighted-sum scalarization can always be rewritten in the anchored form treated in \cite{Fornasier2026}, without assigning any intrinsic privileged role to one objective.

\subsection{Exceptional sets on the simplex}

In the multi-parameter version of the Trade-off Invariance Principle \cite[Theorem~4.1]{Fornasier2026}, exceptional sets are negligible in the parameter space $\alpha\in(0,+\infty)^{N-1}$. Since in the present paper we use the simplex $\Delta_N^\circ$ as the parameter space for symmetric weighted-sum scalarizations, we now check that negligibility is preserved under the corresponding change of coordinates.

\begin{lemma}[Preservation of negligible sets on the simplex]
\label{lem:negligible-normalization}
Let $E\subset(0,+\infty)^{N-1}$ be Lebesgue negligible. Fix $k\in\{1,\dots,N\}$, and define
\begin{equation*}
    E_k^\Delta
    \coloneqq
    \biggl\{
        \lambda\in\Delta_N^\circ:
        \biggl(\frac{\lambda_i}{\lambda_k}\biggr)_{i\neq k}\in E
    \biggr\}.
\end{equation*}
Then $E_k^\Delta$ is negligible with respect to the natural $(N-1)$-dimensional Lebesgue measure on $\Delta_N^\circ$.
\end{lemma}

\begin{proof}
It is enough to prove the statement for $k=N$, since the general case follows by relabelling the objectives.
After this relabelling, the coordinates $(\alpha_i)_{i\neq k}$ are indexed as $(\alpha_1, \ldots, \alpha_{N-1})$.
Define $\Psi\colon(0,+\infty)^{N-1}\to\Delta_N^\circ$ by
\begin{equation*}
    \Psi(\alpha_1,\dots,\alpha_{N-1})
    \coloneqq
    \frac{1}{1+\sum_{i=1}^{N-1}\alpha_i}
    (\alpha_1,\dots,\alpha_{N-1},1).
\end{equation*}
Then $\Psi$ is smooth, hence locally Lipschitz, and for every
$\alpha\in(0,+\infty)^{N-1}$ one has
\begin{equation*}
    \frac{\Psi_i(\alpha)}{\Psi_N(\alpha)}=\alpha_i,
    \qquad i=1,\dots,N-1.
\end{equation*}
Therefore, we conclude that $E_N^\Delta=\Psi(E)$.
Since $E$ is Lebesgue negligible in $\R^{N-1}$ and $\Psi$ is locally Lipschitz, the image $\Psi(E)$ is negligible with respect to the induced $(N-1)$-dimensional Lebesgue measure on $\Delta_N^\circ$ (see, e.g., \cite[Corollary~2.10.11]{Federer1996}).
This proves the claim.
\end{proof}

\begin{remark}
\Cref{lem:normalization,lem:negligible-normalization} allow us to transfer the Trade-off Invariance Principle of \cite[Theorem~4.1]{Fornasier2026} from the anchored family $f_k+\sum_{i\neq k}\alpha_i f_i$
to the symmetric weighted-sum scalarization \eqref{eq:weighted-sum} parametrized by $\lambda\in\Delta_N^\circ$.
Choosing one objective as reference is only a change of coordinates on the simplex, and negligible exceptional sets remain negligible under this parametrization.
\end{remark}

\section{Generic objective-space uniqueness for weighted scalarizations} \label{section:generic-uniqueness}

We now derive the multi-objective counterpart of the Trade-off Invariance Principle.
The main point is that, outside a negligible set of weights in the simplex, possible non-uniqueness for the scalarized problem \eqref{eq:weighted-sum} does not produce different objective vectors.
Since minimizers of weighted-sum scalarizations with strictly positive weights are Pareto optimal, the following result shows that generically all minimizers correspond to the same nondominated objective vector, even when the set of efficient solutions in the decision space is not a singleton.

\begin{theorem}[Objective-space invariance for minimizers]
\label{thm:objective-invariance-minimizers}
{Assume that $\operatorname{dom}f\neq\varnothing$, and} let $H_\lambda$ and $\M_\lambda$ be as in \eqref{eq:weighted-sum} and \eqref{eq:set_minimizers}, respectively.
Then, for almost every $\lambda\in\Delta_N^\circ$, the implication
\begin{equation*}
    u_1,u_2\in\M_\lambda
    \quad\Longrightarrow\quad
    f(u_1)=f(u_2)
\end{equation*}
holds.
Equivalently, for almost every $\lambda\in\Delta_N^\circ$, the set $f(\M_\lambda)$ contains at most one point.
\end{theorem}

\begin{proof}
Let us fix $\ell\in\{1,\dots,N\}$. By \cref{lem:normalization}, if we choose $f_\ell$ as reference, we can rewrite the scalarized functional \eqref{eq:weighted-sum} in the form
\begin{equation*}
    \widetilde H_\alpha^{(\ell)}
    =
    f_\ell+\sum_{i\neq \ell}\alpha_i f_i,
    \qquad
    \alpha_i=\frac{\lambda_i}{\lambda_\ell},
    \quad i\neq \ell.
\end{equation*}
By \cite[Corollary~4.1]{Fornasier2026}, there exists a negligible set
\begin{equation*}
    E_\ell\subset(0,+\infty)^{N-1}
\end{equation*}
such that, for every $\alpha\notin E_\ell$, whenever the minimizer set of $\widetilde H_\alpha^{(\ell)}$ is nonempty, every non-reference component $f_i$, $i\neq \ell$, is constant on that minimizer set. 
{Since $\operatorname{dom}f\neq\varnothing$, the functional $\widetilde H_\alpha^{(\ell)}$ is finite at some point. Therefore, if its minimizer set is nonempty, its minimum is finite. Recalling that all coefficients are strictly positive, every objective $f_i$ must be finite on the minimizer set.}
From the fact that $\widetilde H_\alpha^{(\ell)}$ itself is constant on its minimizer set, it follows by subtraction that also the reference component $f_\ell$ is constant there. Hence, for every $\alpha \not \in E_\ell$, the whole vector-valued map $f$ is constant on the minimizer set of $\widetilde H_\alpha^{(\ell)}$.
Let us introduce 
\begin{equation}  \label{eq:def_negligible_ell}
    E_\ell^\Delta
    \coloneqq
    \biggl\{
        \lambda\in\Delta_N^\circ:
        \biggl(\frac{\lambda_i}{\lambda_\ell}\biggr)_{i\neq \ell}\in E_\ell
    \biggr\}.
\end{equation}
By \cref{lem:negligible-normalization}, the set $E_\ell^\Delta$ is negligible in $\Delta_N^\circ$.
Finally, let $\lambda\in\Delta_N^\circ\setminus E_\ell^\Delta$.
Then the minimizer sets of $H_\lambda$ and of $\widetilde H_\alpha^{(\ell)}$ coincide by \cref{lem:normalization}, and therefore $f$ is constant on $\M_\lambda$.
This proves the claim.
\end{proof}

We now pass from exact to approximate minimization. The relevant issue is whether the limiting objective vector along minimizing sequences is uniquely determined by the scalarization, or may depend on the particular minimizing sequence.

\begin{theorem}[Objective-space invariance for minimizing sequences]
\label{thm:objective-invariance-min-seq}
{Assume that $\operatorname{dom}f\neq\varnothing$, and} let $H_\lambda$ and $V(\lambda)$ be as in \eqref{eq:weighted-sum} and \eqref{eq:value-function}, respectively. Then, for almost every $\lambda\in\Delta_N^\circ$, the following property holds: if $V(\lambda)>-\infty$, then there exists a vector $y_\lambda=(y_{\lambda,1},\dots,y_{\lambda,N})\in[-\infty,+\infty]^N$ such that, for every minimizing sequence $(u_k)_{k\in\NN}$ for $H_\lambda$,
\begin{equation*}
    \lim_{k\to\infty} f_i(u_k)=y_{\lambda,i}
    \qquad
    \text{for every } i=1,\dots,N.
\end{equation*}
In particular, the limiting objective vector is uniquely determined by $\lambda$ and does not depend on the chosen minimizing sequence.
\end{theorem}

\begin{proof}
Fix $\ell\in\{1,\dots,N\}$. By \cref{lem:normalization}, choosing $f_\ell$ as reference rewrites the scalarized functional \eqref{eq:weighted-sum} in the form
\begin{equation*}
    \widetilde H_\alpha^{(\ell)}
    =
    f_\ell+\sum_{i\neq \ell}\alpha_i f_i,
    \qquad
    \alpha_i=\frac{\lambda_i}{\lambda_\ell},
    \quad i\neq \ell.
\end{equation*}
By \cite[Theorem~4.1]{Fornasier2026}, there exists a negligible set $E_\ell\subset(0,+\infty)^{N-1}$ such that, for every $\alpha\notin E_\ell$, if the infimum of $\widetilde H_\alpha^{(\ell)}$ is larger than $-\infty$, then every non-reference component $f_i$, $i\neq \ell$, has a limit along every minimizing sequence, and this limit is independent of the chosen minimizing sequence.
Let $E_\ell^\Delta$ be defined as in \cref{eq:def_negligible_ell}.
By \cref{lem:negligible-normalization}, the set $E_\ell^\Delta$ is negligible in $\Delta_N^\circ$. Now set $E^\Delta \coloneqq \bigcup_{\ell=1}^N E_\ell^\Delta$.
Then $E^\Delta$ is negligible.\\
Let $\lambda\in\Delta_N^\circ\setminus E^\Delta$ be such that $V(\lambda)>-\infty$, and let $(u_k)_{k\in\NN}$ be a minimizing sequence for $H_\lambda$. Fix $i\in\{1,\dots,N\}$ and choose $\ell\neq i$. Since $\lambda\notin E_\ell^\Delta$, the previous conclusion applies to the anchored functional $\widetilde H_\alpha^{(\ell)}$. By \cref{lem:normalization}, minimizing sequences for $H_\lambda$ and for $\widetilde H_\alpha^{(\ell)}$ coincide. Hence $\lim_{k\to\infty} f_i(u_k)$ exists in $[-\infty,+\infty]$ and is independent of the chosen minimizing sequence.
Since $i$ is arbitrary, all objective components have uniquely determined limits. Denoting them by $y_{\lambda,i}$ yields the desired vector $y_\lambda$.
\end{proof}

In the bi-objective case, the simplex is one-dimensional, and the genericity statements above reduce to the one-parameter Trade-off Invariance Principle \cite[Theorem~1.2]{Fornasier2026}.

\begin{corollary}[The bi-objective case]
\label{cor:bi-objective}
Assume $N=2$, and let $H_\lambda$, $V(\lambda)$, and $\M_\lambda$ be as in \eqref{eq:weighted-sum}, \eqref{eq:value-function}, and \eqref{eq:set_minimizers}. Then the exceptional set of weights $\lambda\in\Delta_2^\circ$ for which the conclusions of \cref{thm:objective-invariance-minimizers,thm:objective-invariance-min-seq} may fail is at most countable. Equivalently, outside at most countably many ratios $\alpha={\lambda_2}/{\lambda_1}\in(0,+\infty)$, both conclusions hold for the scalarization $f_1+\alpha f_2$.
\end{corollary}

\begin{proof}
For $N=2$, the simplex $\Delta_2^\circ$ is in one-to-one correspondence with $(0,+\infty)$ through the ratio $\alpha=\lambda_2/\lambda_1$. Moreover, by \cref{lem:ray-invariance}, the weighted-sum scalarization \eqref{eq:weighted-sum} depends only on this ratio, since
$H_\lambda$ is equivalent, up to multiplication by a positive constant, to the one-parameter family $f_1+\alpha f_2$.
The minimizing-sequence statement then follows from \cite[Theorem~1.2]{Fornasier2026}, and the minimizer statement follows from \cite[Corollary~3.1]{Fornasier2026}. In both cases, the exceptional set is at most countable. 
\end{proof}

\begin{remark} \label{rem:objective-space-single-valued}
In the terminology of \eqref{eq:weighted-sum} and \eqref{eq:set_minimizers}, \cref{thm:objective-invariance-minimizers} says that weighted scalarizations are generically single-valued in objective space. Thus, although $\M_\lambda$ may contain several distinct elements of $U$, for almost every weight vector all such elements represent the same trade-off:
\begin{equation*}
    u_1,u_2\in\M_\lambda
    \quad\Longrightarrow\quad
    f(u_1)=f(u_2).
\end{equation*}
By \cref{thm:objective-invariance-min-seq}, the analogous statement holds at the level of approximate minimization: for almost every $\lambda$, all minimizing sequences for $H_\lambda$ in \eqref{eq:weighted-sum} determine the same limiting objective vector.
\end{remark}

\subsection{Geometric interpretation in the objective space}

For completeness, we spell out the geometric meaning of the previous results in the objective space; see \cite{Ehrgott2005} for general background on multi-objective optimization.
{We recall that the attainable objective set associated with \eqref{eq:MOO} has been introduced in \cref{eq:def_attainable}, and that} an objective vector is a vector $f(u)=\big(f_1(u),\dots,f_N(u)\big)$ of objective values, with $u\in \operatorname{dom} f$. For a fixed weight vector $\lambda\in\Delta_N^\circ$, the scalarized problem \eqref{eq:weighted-sum} is equivalent to minimizing the linear functional $y\mapsto \lambda\cdot y$ over $Y$, since $V(\lambda)
    =
    \inf_{u\in U}\lambda\cdot f(u)
    =
    \inf_{y\in Y}\lambda\cdot y$.
When $\M_\lambda\neq\varnothing$, the set of objective vectors selected by the scalarization is
\begin{equation*}
    Y_\lambda
    \coloneqq
    f(\M_\lambda)
    =
    \big\{
        y\in Y:
        \lambda\cdot y=V(\lambda)
    \big\}.
\end{equation*}
In this sense, $Y_\lambda$ is the subset of the attainable objective set exposed by the level hyperplane $\{y\in\R^N:\lambda\cdot y=V(\lambda)\}$. Here ``exposed'' means that the linear functional $y\mapsto\lambda\cdot y$ attains its minimum on $Y$ precisely on that subset.

\begin{remark}
\label{rem:geometric-interpretation}
With the notation above, \cref{thm:objective-invariance-minimizers} says that, for almost every $\lambda\in\Delta_N^\circ$, the exposed set $Y_\lambda=f(\M_\lambda)$ contains at most one point. 
Equivalently, for almost every weight vector in the simplex, the supporting hyperplane associated with the scalarization \eqref{eq:weighted-sum} does not select two (or more) distinct attainable objective vectors.
\end{remark}

\begin{remark}[Relation with classical efficiency notions] \label{rmk:relation_classical}
The geometric picture above is consistent with the classical interpretation of weighted-sum scalarization in multi-objective optimization. Indeed, in the objective-space approach, solutions of weighted-sum scalarizations with strictly positive weights correspond to nondominated objective vectors, while scalarizations with merely nonnegative (and nonzero) weights correspond to weakly nondominated ones. 
{We recall that weighted-sum scalarizations select supported nondominated points, while unsupported points of the Pareto front may not be attained for any choice of the weights. Under suitable convexity assumptions, converse scalarization results are available, see \cite[Chapter~3]{Ehrgott2005}.}
In this terminology, for $\lambda\in\Delta_N^\circ$ the set $Y_\lambda=f(\M_\lambda)$ consists of \emph{nondominated points} \cite[Theorem~3.6]{Ehrgott2005}. 
Hence \cref{thm:objective-invariance-minimizers} may be read as a generic uniqueness statement for the nondominated point exposed by a positive weighted-sum scalarization.
We stress once again that this statement concerns uniqueness among the objective vectors selected by weighted-sum scalarizations, and that it does not imply that by varying the weights we can recover the entire Pareto front.
\end{remark}

\begin{remark}[$L^p$-scalarizations] \label{rmk:Lp-scalarizations}
In view of the discussion in \cref{rmk:relation_classical}, one may be interested in alternative scalarizations that may allow to find further nondominated objective vectors.
To do this, let us assume that the ideal point $z^{\mathrm{id}}\in\R^N$ is finite, where
\begin{equation*}
    z_i^{\mathrm{id}}
    \coloneqq
    \inf_{u\in\operatorname{dom}f} f_i(u),
    \qquad
    i=1,\dots,N,
\end{equation*}
and, for $1\leq p<+\infty$ and $\lambda\in\Delta_N^\circ$ let us consider the $L^p$-scalarization, given by
\begin{equation} \label{eq:scal_Lp}
    H_{\lambda,p}(u)
    \coloneqq
    \left(
        \sum_{i=1}^N
        \lambda_i
        \big(f_i(u)-z_i^{\mathrm{id}}\big)^p
    \right)^{1/p}.
\end{equation}
Since $t\mapsto t^{1/p}$ is strictly increasing for $t\geq0$, the minimization of  $H_{\lambda,p}$ is equivalent to minimizing the weighted sum of the transformed objectives
\begin{equation*}
    g_i(u)
    \coloneqq
    \big(f_i(u)-z_i^{\mathrm{id}}\big)^p,
    \qquad
    i=1,\dots,N.
\end{equation*}
Therefore, we can apply \cref{thm:objective-invariance-minimizers} to $H_{\lambda,p}^p$,
and we deduce that, for almost every $\lambda\in\Delta_N^\circ$, all minimizers of $H_{\lambda,p}$ correspond to the same objective vector.
Furthermore, since all weights are strictly positive, every minimizer is Pareto optimal (see, e.g., \cite[Theorem~4.21]{Ehrgott2005}).
Thus, varying $p$ and $\lambda$ may allow to find nondominated objective vectors that are not selected by the weighted-sum method. Furthermore, we recall that for $p=1$, the scalarization in \cref{eq:scal_Lp} differs from the usual weighted sum only by an additive constant.
Finally, we mention that the limiting case $p=+\infty$---corresponding to the weighted Chebyshev scalarization---is a viable strategy for a full characterization of the weak Pareto front, even in the absence of convexity (see, e.g., \cite[Theorem~4.24]{Ehrgott2005}).
However, since it involves the maximum of the weighted objective deviations rather than their sum, it does not fit the additive framework underlying \cref{thm:objective-invariance-minimizers}, and our result does not directly apply to it.
\end{remark}

\begin{example}[A linear Pareto front]
Let us consider the domain $U=[0,1]$, and let us define the objectives
\begin{equation*}
    t\mapsto f(t) = \big( f_1(t), f_2(t) \big)
    \coloneqq
    \big( t, 1-t \big)
\end{equation*}
for every $t\in U$.
Then, the attainable objective set
\begin{equation*}
    Y=f(U)
    =
    \big\{
        (t,1-t):
        t \in [0,1]
    \big\}
\end{equation*}
is a linear Pareto front. For $\lambda\in\Delta_2^\circ$, the weighted-sum scalarization is
\begin{equation*}
    t\mapsto  H_\lambda(t)
    \coloneqq 
    \lambda_1t+\lambda_2(1-t)
    =
    \lambda_2+(\lambda_1-\lambda_2)t.
\end{equation*}
If $\lambda_1>\lambda_2$, its unique minimizer is $t=0$, while, if $\lambda_1<\lambda_2$, its unique minimizer is $t=1$. On the other hand, for the weight vector $\lambda
    =
    \left(\frac{1}{2},\frac{1}{2}\right)$
the scalarization is constant on $U$, and hence every point of the Pareto front is selected. Thus, on the one hand, the exceptional weight is negligible in $\Delta_2^\circ$, consistently with \cref{thm:objective-invariance-minimizers}. On the other hand, missing this exceptional weight vector may be significant from a computational viewpoint.
In this example, the $L^p$-scalarizations discussed in \cref{rmk:Lp-scalarizations} are particularly useful. We first notice that the ideal point of this multi-objective problem is $z^{\mathrm{id}}=(0,0)$. 
For $1<p<+\infty$, the minimization of the $L^p$-scalarization introduced in \eqref{eq:scal_Lp} is equivalent to the optimization problem
\begin{equation*}
  \min_{[0,1]} H_{\lambda,p}^p, \qquad [0,1] \ni t \mapsto H_{\lambda,p}^p(t) \coloneqq
    \lambda_1t^p+\lambda_2(1-t)^p.
\end{equation*}
The unique minimizer of $H_{\lambda,p}$ is
\begin{equation*}
    t_\lambda
    =
    \frac{\lambda_2^{1/(p-1)}}
    {\lambda_1^{1/(p-1)}+\lambda_2^{1/(p-1)}}
    \in(0,1),
\end{equation*}
and, as $\lambda$ varies in $\Delta_2^\circ$, the point $f(t_\lambda)$ ranges over the relative interior of the Pareto front. 
Hence, in this example, $L^p$-scalarizations with $p>1$ allow us to individually select the nondominated objective vectors in the relative interior of the Pareto front. This should be compared with the weighted-sum scalarizations, which highlight such points only simultaneously at the exceptional weight.
\end{example}

\subsection{Asymptotically exposed objective vectors}

The minimizing-sequence result admits a natural geometric interpretation in the closure of the attainable objective set. In this subsection, we restrict to the finite-valued case $Y\subset\R^N$, so that convergence in the objective space is understood in the usual Euclidean sense.

\begin{defn}[Asymptotically exposed objective vectors]
\label{def:aexp}
Let $\lambda\in\Delta_N^\circ$. We define the set of asymptotically exposed objective vectors by
\begin{equation} \label{eq:def_aexp}
    \operatorname{AExp}_\lambda(Y)
    \coloneqq
    \big\{
        y\in\R^N:
        \exists (y_k)_{k\in\NN}\subset Y
        \text{ with }
        y_k\to y
        \text{ and }
        \lambda\cdot y_k\to V(\lambda)
    \big\}.
\end{equation}
\end{defn}

In other words, $\operatorname{AExp}_\lambda(Y)$ consists of those limit points in $\overline Y$ that are reached by objective vectors associated with minimizing sequences for the scalarized problem \eqref{eq:weighted-sum}.

\begin{proposition}[Asymptotic exposure and the closure of $Y$]
\label{prop:aexp-closure}
Assume that $Y\subset\R^N$. Then, for every $\lambda\in\Delta_N^\circ$,
\begin{equation*}
    \operatorname{AExp}_\lambda(Y)
    =
    \big\{
        y\in\overline Y:
        \lambda\cdot y=V(\lambda)
    \big\}.
\end{equation*}
\end{proposition}

\begin{proof}
Let $y\in\operatorname{AExp}_\lambda(Y)$. By definition, there exists a sequence $(y_k)_{k\in\NN}\subset Y$ such that $y_k\to y$ and $\lambda\cdot y_k\to V(\lambda)$ as $k\to\infty$. Since $Y\subset\R^N$ and the map $y\mapsto \lambda\cdot y$ is continuous on $\R^N$, passing to the limit gives $\lambda\cdot y=V(\lambda)$. Hence
\begin{equation*}
    \operatorname{AExp}_\lambda(Y)
    \subset
    \big\{
        y\in\overline Y:
        \lambda\cdot y=V(\lambda)
    \big\}.
\end{equation*}
Conversely, let $y\in\overline Y$ be such that $\lambda\cdot y=V(\lambda)$. Since $y\in\overline Y$, there exists a sequence $(y_k)_{k\in\NN}\subset Y$ with $y_k\to y$. By continuity of $y\mapsto \lambda\cdot y$, we obtain $\lambda\cdot y_k\to \lambda\cdot y=V(\lambda)$. Therefore $y\in\operatorname{AExp}_\lambda(Y)$.
\end{proof}

\begin{corollary}[Generic uniqueness of asymptotically exposed objective vectors]
\label{cor:aexp-unique}
Assume that $Y\subset\R^N$. Then, for almost every $\lambda\in\Delta_N^\circ$ such that $V(\lambda)>-\infty$, the set $\operatorname{AExp}_\lambda(Y)$ contains at most one point.
\end{corollary}

\begin{proof}
Let $\lambda\in\Delta_N^\circ$ be such that the conclusion of \cref{thm:objective-invariance-min-seq} holds and $V(\lambda)>-\infty$. Take $y,z\in\operatorname{AExp}_\lambda(Y)$. By \cref{def:aexp}, there exist sequences $(y_k)_{k\in\NN},(z_k)_{k\in\NN}\subset Y$ such that $y_k\to y$, $z_k\to z$, and $\lambda\cdot y_k\to V(\lambda)$, $\lambda\cdot z_k\to V(\lambda)$.
Since $Y=f(U)$, we may write $y_k=f(u_k)$ and $z_k=f(v_k)$ for suitable sequences $(u_k)_{k\in\NN},(v_k)_{k\in\NN}\subset U$. The identities $\lambda\cdot y_k\to V(\lambda)$ and $\lambda\cdot z_k\to V(\lambda)$ mean that $(u_k)_{k\in\NN}$ and $(v_k)_{k\in\NN}$ are minimizing sequences for $H_\lambda$. Hence \cref{thm:objective-invariance-min-seq} yields $    \lim_{k\to\infty} f(u_k)  = \lim_{k\to\infty} f(v_k)$.
Therefore $y=z$. This proves that $\#\operatorname{AExp}_\lambda(Y)\leq 1$.
\end{proof}

\begin{remark}[Asymptotic exposure and nondominance]
\label{rem:aexp-nondominated}
Assume that $Y\subset\R^N$, let $\lambda\in\Delta_N^\circ$, and let $y\in\operatorname{AExp}_\lambda(Y)$. Then $y$ is nondominated in $\overline Y$. Indeed, if there existed $z\in\overline Y$ with $z_i\le y_i$ for every $i=1,\dots,N$ and $z\neq y$, then, since all components of $\lambda$ are strictly positive, one would have $\lambda\cdot z<\lambda\cdot y=V(\lambda)$, contradicting \cref{prop:aexp-closure}. Thus, exact minimizers expose nondominated points of $Y$, while minimizing sequences expose nondominated points of $\overline Y$. In this sense, \cref{cor:aexp-unique} gives a generic uniqueness statement for the asymptotically exposed---hence asymptotically nondominated---objective vector selected by the scalarization.
\end{remark}

\subsection*{Acknowledgements}

J.K.~and A.S. acknowledge support from the ERC Advanced Grant NEITALG, grant agreement No.~101198055 (P.I.: Prof.~Massimo Fornasier).
A.S.~acknowledges partial support from INdAM--GNAMPA.
{The authors want to thank two anonymous referees for their helpful comments, which improved the overall quality of the paper.}

\bigskip
\begin{center}
  \FundingLogos
  
  \vspace{0.5em}
  \begin{tcolorbox}\centering\small
    Funded by the European Union. Views and opinions expressed are however those of the author(s) only and do not necessarily reflect those of the European Union or the European Research Council Executive Agency. Neither the European Union nor the granting authority can be held responsible for them.
    This project has received funding from the European Research Council (ERC) under the European Union’s Horizon Europe research and innovation programme (grant agreement No.~101198055, project acronym NEITALG).
    
  \end{tcolorbox}
\end{center}

\bibliographystyle{abbrv}
\bibliography{biblio}

\end{document}